\documentclass[12pt]{amsart}
\usepackage{amsmath,amsfonts,amsthm,amsopn,amssymb,mathtools,stmaryrd}
\usepackage{cite,marginnote}
\usepackage{pdfsync}

\usepackage{color,enumitem,graphicx}
\usepackage[colorlinks=true,urlcolor=blue,
citecolor=red,linkcolor=blue,linktocpage,pdfpagelabels,
bookmarksnumbered,bookmarksopen]{hyperref}
\usepackage[english]{babel}
\usepackage[left=2.3cm,right=2.3cm,top=2.7cm,bottom=2.7cm]{geometry}

\usepackage[]{xcolor}

\def\sideremark#1{\ifvmode\leavevmode\fi\vadjust{\vbox to0pt{\vss
 \hbox to 0pt{\hskip\hsize\hskip1em
 \vbox{\hsize2.1cm\tiny\raggedright\pretolerance10000
  \noindent #1\hfill}\hss}\vbox to15pt{\vfil}\vss}}}%

\numberwithin{equation}{section}

\makeindex
\newtheorem{theorem}{Theorem}[section]
\newtheorem{proposition}[theorem]{Proposition}
\newtheorem{lemma}[theorem]{Lemma}
\newtheorem{remark}[theorem]{Remark}
\newtheorem{example}[theorem]{Example}
\newtheorem{corollary}[theorem]{Corollary}
\newtheorem{definition}[theorem]{Definition}

\newcommand{\RN}{\mathbb R^N}

\newcommand{\s}{\section}

\newcommand{\lab}{\label}
\newcommand{\bt}{\begin{theorem}}
\newcommand{\et}{\end{theorem}}
\newcommand{\bl}{\begin{lemma}}
\newcommand{\el}{\end{lemma}}
\newcommand{\bd}{\begin{definition}}
\newcommand{\ed}{\end{definition}}
\newcommand{\bc}{\begin{corollary}}
\newcommand{\ec}{\end{corollary}}
\newcommand{\bp}{\begin{proof}}
\newcommand{\ep}{\end{proof}}
\newcommand{\bx}{\begin{example}}
\newcommand{\ex}{\end{example}}
\newcommand{\bi}{\begin{exercise}}
\newcommand{\ei}{\end{exercise}}
\newcommand{\bo}{\begin{proposition}}
\newcommand{\eo}{\end{proposition}}
\newcommand{\br}{\begin{remark}}
\newcommand{\er}{\end{remark}}
\newcommand{\be}{\begin{equation}}
\newcommand{\ee}{\end{equation}}
\newcommand{\ba}{\begin{align}}
\newcommand{\ea}{\end{align}}
\newcommand{\bn}{\begin{enumerate}}
\newcommand{\en}{\end{enumerate}}
\newcommand{\bg}{\begin{align*}}
\newcommand{\bcs}{\begin{cases}}
\newcommand{\ecs}{\end{cases}}

\newcommand{\bean}{\begin{eqnarray*}}
\newcommand{\eean}{\end{eqnarray*}}

\renewcommand{\leq}{\leqslant}
\renewcommand{\geq}{\geqslant}

\title[Normalized solutions to quasilinear Schr\"odinger equations]{Normalized solutions of quasilinear Schr\"odinger equations with a general nonlinearity}

\author[T. Deng]{Ting Deng}
\author[M. Squassina]{Marco Squassina}
\author[J. J. Zhang]{Jianjun Zhang}
\author[X.X. Zhong]{Xuexiu Zhong}

\address[T.~Deng]{\newline\indent College of Mathematics and Statistics
\newline\indent
Chongqing Jiaotong University
\newline\indent
Chongqing 400074, PR China}
\email{\href{mailto:1666106471@qq.com}{1666106471@qq.com}}

\address[M. Squassina]{\newline\indent Dipartimento di Matematica e Fisica
\newline\indent
Universit\`a Cattolica del Sacro Cuore
\newline\indent
Via dei Musei 41, Brescia, Italy}\email{\href{marco.squassina@unicatt.it}{marco.squassina@unicatt.it}}

\address[J.J.~Zhang]{\newline\indent College of Mathematics and Statistics
\newline\indent
Chongqing Jiaotong University
\newline\indent
Chongqing 400074, PR China}
\email{\href{mailto:zhangjianjun09@tsinghua.org.cn}{zhangjianjun09@tsinghua.org.cn}}

\address[X.X.~Zhong]{\newline\indent South China Research Center for Applied Mathematics and Interdisciplinary Studies
\newline\indent
South China Normal University
\newline\indent
Guangzhou 510631, P. R. China}
\email{\href{mailto:zhongxuexiu1989@163.com}{zhongxuexiu1989@163.com}}

\thanks{Marco Squassina is supported by Gruppo Nazionale per l'Analisi Ma\-te\-ma\-ti\-ca, la Probabilit\`a e le loro Applicazioni. X.X.~Zhong is partially supported by the NSFC (No.12271184), Guangdong Basic and Applied Basic Research Foundation (2021A1515010034), Guangzhou Basic and Applied Basic Research Foundation (202102020225). J. J. Zhang was supported by NSFC (No.11871123).}

\subjclass[2010]{35J62, 35B40, 35B09}
\date{}
\keywords{Quasilinear Schr\"odinger equations, Normalized solutions, Mass critical exponent.}

\begin{document}

\begin{abstract}
We are concerned with solutions of the following quasilinear Schr\"odinger equations
\begin{eqnarray*}
-{\mathrm{div}}\left(\varphi^{2}(u) \nabla u\right)+\varphi(u) \varphi^{\prime}(u)|\nabla u|^{2}+\lambda u=f(u), \quad x \in \mathbb{R}^{N}
\end{eqnarray*}
with prescribed mass
$$
\int_{\mathbb{R}^{N}} u^{2} \mathrm{d}x=c,
$$
where $N\ge 3, c>0$, $\lambda \in \mathbb{R}$ appears as the Lagrange multiplier and $\varphi\in C ^{1}(\mathbb{R} ,\mathbb{R}^{+})$. The nonlinearity $f \in C\left ( \mathbb{R}, \, \mathbb{R} \right )$ is allowed to be mass-subcritical,
mass-critical and mass-supercritical at origin and infinity. Via a dual approach, the fixed point index and a global branch approach, we establish the existence of normalized solutions to the problem above. The results  extend  previous results by L. Jeanjean, J. J. Zhang and X.X. Zhong to the quasilinear case.
		
\end{abstract}

\maketitle

\begin{center}
	\begin{minipage}{9.5cm}
		\small
		\tableofcontents
	\end{minipage}
\end{center}

\medskip

\s{Introduction}
\renewcommand{\theequation}{1.\arabic{equation}}
\subsection{Background and motivation} This paper is concerned with the quasilinear Schr\"odinger equation of the form
		\begin{equation}\lab{p1}
			-{\mathrm{div}}\left(\varphi^{2}(u) \nabla u\right)+\varphi(u) \varphi^{\prime}(u)|\nabla u|^{2}+\lambda u=f(u), \quad x \in \mathbb{R}^{N},
		\end{equation}
which is relevant to solitary wave solutions of quasilinear Schr\"odinger equation
		\begin{equation}\lab{a1}
			i z_{t}=-\Delta z+W(x) z-h\left(|z|^{2}\right) z-\Delta\left(l\left(|z|^{2}\right)\right) l^{\prime}\left(|z|^{2}\right) z, \quad x \in \mathbb{R}^{N},
		\end{equation}
		where $z:\mathbb{R}\times \mathbb{R}^{N}\to \mathbb{C}$, $W:\mathbb{R}^{N}\to \mathbb{R}$ is a given potential and $l$, $h$ are real functions.
		Quasilinear equations of the form (\ref{a1}) appear naturally in mathematical physics and have been accepted as models of several physical phenomena corresponding to various types of nonlinear terms $l$. When $l\left ( s \right ) = s$, that is, $\varphi^2(u)=1+2u^2$, Equation \eqref{a1} arises in plasma physics for the superfluid film equation (see \cite{A11}), while when $l(s)=(1+s)^{1 / 2}$, that is, $\varphi^2(u)=1+\frac{u^2}{2\left(1+u^2\right)}$, it models the self-channeling of a high-power ultrashort laser in matter (see \cite{LSS}). Such equations were also showed in plasma physics and fluid mechanics \cite{A13}, in dissipative quantum mechanics \cite{A8}, in the theory of Heisenberg ferromagnetism and magnons \cite{A12,A18}, and in condensed matter theory \cite{A17}. For more details, we refer to \cite{A14,A15}.
\vskip0.1in		
In the last decades, quasilinear Schr\"odinger equations have received a considerable attentions from researchers on the existence of nontrivial solutions. In \cite{A4}, M. Colin and L. Jeanjean considered the following quasilinear Schr\"odinger equations
\begin{equation}\label{cj}
-\Delta u-\Delta\left(u^2\right) u=g(x, u), \quad x \in \mathbb{R}^N .
\end{equation}
 Thanks to a dual approach, they transform equation \eqref{cj} into a semilinear elliptic equation. By using the variational method, nontrivial solutions were obtained. In \cite{A6}, E. Gloss adopted the dual approach in \cite{A4} to study the semiclassical states of quasilinear elliptic equations
 $$
-\varepsilon^2 \Delta u-\varepsilon^2 \Delta\left(u^2\right) u+V(x) u=h(u), \quad u>0 \text { in } \mathbb{R}^N.
$$
Via the penalized argument as in \cite{byeon}, she established the existence and concentration of positive solutions without some growth conditions such as Ambrosetti-Rabinowitz. Different from \cite{A4}, Y. Shen and Y. Wang \cite{A19} introduced a new variable replacement to study the existence of nontrivial solutions for generalized quasilinear Schr\"odinger equations. By using the dual approach in \cite{A19}, Y. Deng, S. Peng and S. Yan \cite{A5} considered a generalized quasilinear Schr\"odinger equation with critical growth. By virtue of variational approaches, they obtained the existence of positive solutions. In \cite{A1}, C. O. Alves, Y. Wang and Y. Shen considered the following quasilinear Schr\"odinger equation with one parameter
\begin{equation}\label{cws}
-\Delta u+V(x) u+\frac{\kappa}{2}\Delta(u^2) u=l(u), x \in \mathbb{R}^N,
\end{equation}
where $\kappa$ is allowed to be positive. When $\kappa<0$, one can use the dual approaches in \cite{A4,A19} to deal with the existence of nontrivial solutions for equation \eqref{cws}. The case of $\kappa>0$ becomes more complicated. By adopting a truncation approach, the authors obtained the existence of nontrivial classical solution for $\kappa>0$ small enough. For more related results on quasilinear elliptic equations, we refer to \cite{LiuX1,LiuX3,hzz} and references therein. For a direct approach via nonsmooth critical point theory
in cases where the change of variable is not usable, see \cite{pellacci}.
\vskip0.1in
In this paper, we focus on normalized solutions of equation \eqref{q1} with prescribed mass, that is to find $(\lambda,u)$ such that
		\begin{equation}\lab{q1}
			\left\{
			\begin{array}{ll}
				-\mathrm{div}\left(\varphi^{2}(u) \nabla u\right)+\varphi(u) \varphi^{\prime}(u)|\nabla u|^{2}+\lambda u=f(u), \quad x \in \mathbb{R}^{N}, \\
				\displaystyle\int_{\mathbb{R}^{N}}\left | u \right | ^{2} \mathrm{d}x=c ,
			\end{array}
			\right.
		\end{equation}
		where $N \ge 3$, $f \in C\left ( \mathbb{R},  \, \mathbb{R} \right )$, $\varphi: \mathbb{R}\to \mathbb{R}^{+} $ is a $C^{1}$ nondecreasing function with respect to $\left | s \right |, c> 0$ is a given mass, $\lambda \in \mathbb{R}$ appears as a Lagrange multiplier. From the view of physics, prescribed mass represents the law of conservation of mass, so it is pretty meaningful to investigate the existence of normalized solutions. If $\varphi\equiv1$, problem \eqref{q1} is reduced to the following scalar field equation with prescribed mass
		\begin{equation}\lab{qs1}
			\left\{
			\begin{array}{ll}
				-\Delta u+\lambda u=f(u), \quad x \in \mathbb{R}^{N}, \\
				\displaystyle\int_{\mathbb{R}^{N}}\left | u \right | ^{2} \mathrm{d}x=c .
			\end{array}
			\right.
		\end{equation}
One classical argument to deal with such problems is to find critical points $u\in H^1(\RN)$ of the energy functional
\begin{equation*}
  \mathcal{J}(u) = \frac{1}{2}\int_{\RN}|\nabla u|^2 \, \mathrm{d}x- \int_{\RN}F(u) \, \mathrm{d}x, \quad \mbox{where} \quad F(s):=\int_0^s f(t)\mathrm{d}t,
\end{equation*}
subject to the constraint
$$ S_c:=\{u\in H^1(\RN): \|u\|_{L^2(\RN)}^{2}=c\}.$$
When $f(s)=|s|^{p-2}s$ with $p\in(2,2N/(N-2))$, thanks to the well-known Gagliardo-Nirenberg inequality, it is known that the energy functional $\mathcal{J}$ is bounded from below on $S_a$ for any $a>0$ if $p<2+4/N$ and unbounded from below for any $a>0$ if $p>2+4/N$. $2+4/N$ is called the mass-critical exponent. In \cite{Stuart1981,Stuart1989}, C. A. Stuart considered problem \eqref{qs1} with a mass-subcritical nonlinearity and obtained the existence of normalized solutions by seeking a global minimizer of the energy functional on $S_c$. In the mass supercritical case, since there exists no global minimizer of the associated energy functional restricted on the constraint $S_c$, it seems that the arguments in dealing with the mass subcritical case do not work. In \cite{Jeanjean1997}, L. Jeanjean imposed the following global condition on $f$.
\begin{itemize}
\item [($H$)] $\exists(\alpha, \beta) \in \mathbb{R} \times \mathbb{R}$ satisfying
$$
2+\frac{4}{N}<\alpha \leq \beta<\frac{2 N}{(N-2)_+}
$$
such that
$$
\alpha F(s) \leq f(s) s \leq \beta F(s).
$$
\end{itemize}
The assumption $(H)$ is used to guarantee the mountain pass geometry of $\tilde{\mathcal{J}}$ on $S_c\times\mathbb{R}$, where $\tilde{\mathcal{J}}(u,s)=\mathcal{J}(e^{sN/2}u(e^s\cdot))$. Then a bounded Palais-Smale sequence of $\mathcal{J}$ restricted on $S_c$ is obtained. Via some compactness arguments, the author shows the existence of normalized solutions to problem \eqref{qs1}. More recently, L. Jeanjean, J. Zhang and X. Zhong \cite{A10} introduced a new and non-variational approach to deal with problem \eqref{qs1} in the mass-subcritical, mass-critical and mass-supercritical case in one unified way. In particular, the assumption $(H)$ can be removed. For more related results on problem \eqref{qs1}, we refer to \cite{A10} and references therein.
\vskip0.1in
When $\varphi^2(u)=1+2u^2$, problem \eqref{p1} reads as the following form
\begin{equation}\label{jlw}
-\Delta u-u \Delta\left(u^2\right)+\lambda u=f(u) \quad \text { in } \mathbb{R}^N .
\end{equation}
Different to the case $\varphi\equiv1$, $4+4/N$ was proved by M. Colin, L. Jeanjean and M. Squassina \cite{CJS} to be the mass-critical exponent for problem \eqref{jlw}.
In \cite{A9}. L. Jeanjean, T. Luo and Z.-Q. Wang considered problem \eqref{jlw} with a mass-subcritical nonlinearity $f(u)=|u|^{p-2}u$, $p\in(2+4/N,4+4/N)$ and prove the existence of two solutions if the prescribed $L^{2}$-norm is large enough. In \cite{A21}, H. Ye and Y. Yu focused on the existence of normalized solutions of problem \eqref{jlw} in the mass-critical case. For the mass-supercritical case, we refer to \cite{A16}, where H. Li and W. Zou adopted one perturbation argument as in \cite{LiuX1} to get the existence of the
existence of ground state normalized solutions to problem \eqref{jlw}. By applying the index theory, infinitely many normalized solutions were also obtained. Moreover, they also established the concentration behavior of ground state solutions in the mass-critical case. In \cite{A22}, via the dual approach in \cite{A4} and genus theory, L. Zhang, Y. Li and Z.-Q. Wang constructed multiple normalized solutions of problem \eqref{jlw} with $f(u)=|u|^{p-2}u$ in the mass-subcritical case.
\vskip0.1in
When $\varphi^2(u)=1+\frac{u^2}{2\left(1+u^2\right)}$, equation \eqref{p1} is reduced to the following quasilinear equation
\begin{equation}\label{ytc}
-\Delta u+\lambda u-\left[\Delta\left(1+u^2\right)^{\frac{1}{2}}\right] \frac{u}{2\left(1+u^2\right)^{\frac{1}{2}}}=f(u) \text { in } \mathbb{R}^N.
\end{equation}
In \cite{A20}, under the Berestycki-Lions condition, X. Yang, X. Tang and B. Cheng show that problem \eqref{ytc} admits multiple radial and nonradial normalized solutions in the mass-subcritical case. For more processes on normalized solutions of quasilinear Schr\"odinger equations, we refer to \cite{A9,A16,A21,A22} and the references therein.
		
\subsection{Assumptions and main results} The main purpose of the present paper is to investigate the existence of normalized solutions to problem \eqref{q1} with a relatively general nonlinearity. Throughout this paper, we impose the following assumptions on $\varphi$.

\begin{itemize}
			\item [$\left (\varphi _{0} \right )$]
			$\varphi\in C ^{1} \left ( \mathbb{R} ,\mathbb{R}^{+} \right ) $ is even, $\varphi ' \left ( t \right ) \ge 0 $ for all $t\ge 0 $ and $\varphi\left ( 0 \right ) =1$.
			
			\item [$\left (\varphi _{1} \right )$]   There exists $a^*>0 $ such that
			\begin{equation*}
				\lim_{s \to +  \infty}\varphi \left ( s \right )  =a^* >0 .
			\end{equation*}
            \item [$\left (\varphi _{2} \right )$]
            $\lim\limits_{t\rightarrow\infty}t\varphi'(t)=0$.
\end{itemize}
\br
As a reference model, take $\varphi(t)=\left(1+\frac{t^2}{2\left(1+t^2\right)}\right)^{\frac{1}{2}}$ satisfying $\left (\varphi _{0} \right )$-$\left (\varphi _{2} \right )$.
\er

As for the nonlinearity $f$, we assume that
\begin{itemize}			
			\item [$\left(F_{1}\right)$]
			$f \in C^1[0, + \infty)$, $f(s)>0$ for $s> 0$.
			
			\item [$\left(F_{2}\right)$]
			There exist $\alpha, \beta, \mu_{1}, \mu_{2}>0$ satisfying
			\begin{equation*}
				2<\alpha, \beta<2^{*}:=\frac{2 N}{N-2}
			\end{equation*}
			such that
			\begin{equation*}
				\lim _{t \rightarrow 0^{+}} \frac{f'(t)}{t^{\alpha-2}}=\mu_{1}(\alpha-1)>0 \text { \quad and } \quad \lim _{t \rightarrow+\infty} \frac{f'(t)}{t^{\beta-2}}=\mu_{2}(\beta-1)>0 .
			\end{equation*}
\item [$\left(F_{3}\right)$] There exists no positive radial decreasing classical solution for \\ $-\mathrm{div}\left(\varphi^{2}(u) \nabla u\right)+\varphi(u) \varphi^{\prime}(u)|\nabla u|^{2}=f(u)$ in $\mathbb{R}^N$.
		\end{itemize}
\br
By \cite[Theorem 2.2-(ii)]{A10}, $(F_3)$ holds for $f(t)=|t|^{q-2}t$, $q\in(2,2N/(N-2))$.
\er		

Before to state our result, denote by $U$ the unique positive solution of
\begin{equation}\label{k4}
-\Delta U+U=\mu_1 U^{\alpha-1}~\hbox{in}~\RN,\,\,U(0)=\max_{x\in\mathbb{R}^N}U(x)
\end{equation}			
and $V$ the unique positive solution of
\begin{equation}\label{k5}
-\Delta V+V=\mu_2V^{\beta-1}~\hbox{in}~\RN,\,\,V(0)=\max_{x\in\mathbb{R}^N}V(x).
\end{equation}
It is known that $U,V$ are non-degenerated (See \cite[Proposition 2.1]{A10}).	
\vskip0.1in
Now, our main result reads as follows.
		\begin{theorem}\lab{Th1}
			Let $N \ge 3$ and assume that $\left(\varphi_0\right)$-$\left(\varphi_{2}\right)$, $\left(F_{1}\right)$-$\left(F_{3}\right)$ hold, then we have the following conclusions.
			\begin{itemize}
				\item[(i)] 	{\bf(mass subcritical case)} If $2<\alpha$, $\beta<2+\frac{4}{N}$, then for any given $c>0$, (\ref{h2}) possesses a positive normalized solution $\left(\lambda,  \, v_{\lambda}\right) \in(0,  \, +\infty) \times H_{\text{rad}}^{1}\left(\mathbb{R}^{N}\right)$.
				
				\item[(ii)]  {\bf (exactly mass critical case)} If $\alpha=\beta$ $=2+\frac{4}{N}$,  denote
$$c_\ast:=\min\{\|U\|_2^2,(a^\ast)^N\|V\|_2^2\}, c^\ast:=\max\{\|U\|_2^2,(a^\ast)^N\|V\|_2^2\},$$
then (\ref{h2}) admits at least one positive normalized solution
$(\lambda,  \,  v_{\lambda}) \in(0,  \, +\infty) \times H_{\text {rad }}^{1}(\mathbb{R}^{N})$
provided $c \in(c_\ast,  c^\ast)$ and no positive normalized solution if $c>0$ small or large.
				
				\item[(iii)]  {\bf(at most mass critical case)}
				
				\item[(iii-1)]  If $2<\alpha<\beta=2+\frac{4}{N}$, (\ref{h2}) admits at least one positive normalized solution
$\left(\lambda,  \, v_{\lambda}\right) \in(0,  \, +\infty) \times H_{\text {rad }}^{1}\left(\mathbb{R}^{N}\right)$ if
 $0<c<(a^\ast)^N\|V\|_2^2$ and no positive normalized solution if $c>0$ large.
				
				\item[(iii-2) ] If $2<\beta<\alpha=2+\frac{4}{N}$, (\ref{h2}) admits at least one positive normalized solution
$\left(\lambda,  \, v_{\lambda}\right) \in(0,  \, +\infty) \times H_{\text {rad }}^{1}\left(\mathbb{R}^{N}\right)$
 if $c>\|U\|_2^2$ and no positive normalized solution if $c>0$ small.
				
				\item[(iv) ] {\bf (mixed case)}
				
				\item[(iv-1)]  If $2<\alpha<2+\frac{4}{N}<\beta<2^{*}$, (\ref{h2}) admits at least two distinct positive normalized solutions
$\left(\lambda_{i},  \, v_{\lambda_{i}}\right) \in(0,  \, +\infty) \times H_{\text {rad }}^{1}\left(\mathbb{R}^{N}\right)$ if $c>0$ small and no positive normalized solution
provided $c>0$ large.
				
				\item[(iv-2) ] If $2<\beta<2+\frac{4}{N}<\alpha<2^{*}$, (\ref{h2}) admits at least two distinct positive normalized solutions
$\left(\lambda_{i},  \, v_{\lambda_{i}}\right) \in(0,  \, +\infty) \times H_{\text {rad }}^{1}\left(\mathbb{R}^{N}\right)$ if $c>0$ large and no positive normalized solution
provided $c>0$ small.

				\item[(v) ] {\bf(at least mass critical case)}
				
				\item[(v-1) ] If $2+\frac{4}{N}=\alpha<\beta<2^{*}$, (\ref{h2}) admits at least one positive normalized solution
$\left(\lambda,  \, v_{\lambda}\right) \in(0,  \, +\infty) \times H_{\text {rad }}^{1}\left(\mathbb{R}^{N}\right)$
 if $0<c<\|U\|_{2}^{2}$ and no positive normalized solution provided $c>0$ large.
				
				\item[(v-2)]  If $2+\frac{4}{N}=\beta<\alpha<2^{*}$, (\ref{h2}) admits at least one positive normalized solution
 $\left(\lambda,  \, v_{\lambda}\right) \in(0,  \, +\infty) \times H_{\text {rad }}^{1}\left(\mathbb{R}^{N}\right)$
 if $c>(a^\ast)^N\|V\|_{2}^{2}$ and no positive normalized solution provided $c>0$ small.
				
				\item[(vi) ] {\bf(mass supercritical case)} If $2+\frac{4}{N}<\alpha, \beta<2^{*}$, then for any given $c>0$, (\ref{h2}) admits a normalized positive solution $\left(\lambda, \, v_{\lambda}\right) \in(0,  \, +\infty) \times H_{\text {rad }}^{1}\left(\mathbb{R}^{N}\right)$.
			\end{itemize}
		\end{theorem}		
\subsection{Strategy of this paper} First, by virtue of a dual approach in \cite{A19}, we transform equation \eqref{p1} into a semilinear elliptic equation. Second, for any fixed $\lambda>0$, we show the semilinear elliptic equation obtained admits at least one positive and radially symmetric solution $v_\lambda$. Similarly to \cite{A10}, by using the blow-up argument and a Liouville theorem, as $\lambda\rightarrow 0^{+}$ or $\lambda \rightarrow +\infty$, the asymptotic behaviors of positive solutions are investigated, as well as the $L^2$-norms of $\phi ^{-1} \left ( v_{\lambda}\right ) $. Finally, via the fixed point index and the topological degree, we adopt the similar idea to \cite{A10} to establish a global branch of positive solutions for $\lambda\in(0,\infty)$. By applying a continuity argument, Theorem \ref{Th1} is proved.
\vskip0.1in	
{\bf Notations}
		\begin{itemize}
			\item $L^{p}\left ( \mathbb{R}^{N}  \right )$ denotes the Lebesgue space with norm $\|u\|_p:=\big(\int_{\mathbb{R}^{N}}|u|^p \mathrm{d}x\big )^{1/p}$.
			\item $H^{1}\left ( \mathbb{R}^{N}  \right ) $ denotes the Sobolev space modeled in $L^{2}\left ( \mathbb{R}^{N}  \right ) $ with its usual norm $\|u\|:=\big(\|u\|_2^2+\|\nabla u\|_2^2\big)^{1/2}$.
			\item $C_{r,  \, 0}\left(\mathbb{R}^{N}\right)$ denotes the space of radial continuous functions vanishing at infinity.
		\end{itemize}

		\medskip

	\s{Functional setting and preliminaries}
		\renewcommand{\theequation}{2.\arabic{equation}}
		\hspace*{0.5cm}
To seek normalized solutions of (\ref{q1}), it suffices to find critical points of
		\begin{equation}\lab{g1}
			I\left ( u \right ) =\frac{1}{2} \int_{\mathbb{R}^{N} }^{} \varphi^{2}(u) \left | \nabla u \right | ^{2}\mathrm{d}x-\int_{\mathbb{R}^{N} }^{}F\left ( u \right )\mathrm{d}x
		\end{equation}
on the mass sphere $S_c$. Due to the presence of $\varphi$, some additional difficulties arise. In this paper, we adopt a dual approach in \cite{A19} to overcome them and make a change of variables as follows. Let $(u,\lambda)$ be any solution of problem \eqref{q1}, that is, for any $\tilde{\phi } \in C_{0}^{\infty } \left ( \mathbb{R}^{N} \right )$, there holds
		\begin{equation}\lab{g3}
			\int_{\mathbb{R} ^{N} }^{} \left [ \varphi ^{2}  \left ( u \right ) \nabla u\nabla \tilde{\phi }   + \varphi \left ( u \right ) \varphi ^{\prime } \left ( u \right ) \left | \nabla u \right | ^{2} \tilde{\phi }  +  \lambda u\tilde{\phi }  - f\left ( u \right )\tilde{\phi }  \right   ] \mathrm{d}x=0.
		\end{equation}
If ones take
\begin{equation*}
			v=\phi \left ( u \right )  =\int_{0}^{u} \varphi\left ( t \right ) \mathrm{d}t,
		\end{equation*}
and choose $\tilde{\phi } =\frac{1}{\varphi \left ( u \right ) } \psi $ in equation \eqref{g3} for any $\psi \in C_{0}^{\infty } \left ( \mathbb{R}^{N}  \right )$, then it follows that
		\begin{equation}\lab{h1}
			\int_{\mathbb{R}^{N}}^{} \left [ \nabla v\nabla \psi +\lambda \frac{ \phi^{-1} \left ( v \right )}{\varphi  \left (\phi^{-1} \left ( v \right )  \right ) }\psi - \frac{f\left (\phi^{-1} \left ( v \right )  \right )}{\varphi \left (\phi^{-1} \left ( v \right )  \right )}\psi   \right ]\mathrm{d}x=0.
		\end{equation}		
Then $(v,\lambda)$ is a solution of the following semilinear elliptic problem
\begin{equation}\lab{h2}
				-\Delta v+\lambda \frac{\phi^{-1} \left ( v \right ) }{\varphi \left (\phi^{-1} \left ( v \right )   \right )}
-\frac{f \left (\phi^{-1} \left ( v \right ) \right )}{\varphi \left (\phi^{-1} \left ( v \right  )  \right )} =0 ,  \quad x \in \mathbb{R}^{N}
		\end{equation}
with the prescribed mass
$$
\int_{\mathbb{R}^{N}}\left | \phi^{-1} \left ( v \right )  \right | ^{2} \mathrm{d}x=c.
$$
And corresponding to the energy functional $I$, the energy functional associated with problem \eqref{h2} is defined by
		\begin{equation}\lab{g2}
			J\left ( v \right ) =\frac{1}{2} \int_{\mathbb{R}^{N} }^{}\left  | \nabla v \right |^{2}\mathrm{d}x -\int_{\mathbb{R}^{N} }^{}F\left ( \phi^{-1} \left ( v \right )  \right ) \mathrm{d}x.
		\end{equation}
Since $\varphi$ is a nondecreasing positive function, we get $\left | \phi ^{-1}\left ( s \right )  \right | \le  \left | s \right |$ for any $s$. Moreover, it is clear that $J$ is well defined in $H^{1}\left ( \mathbb{R}^{N}  \right ) $ and of $C^{1}$-class. Therefore, in order to find normalized solutions of equation \eqref{q1}, it is sufficient to turn to consider the existence of normalized solutions to problem \eqref{h2}.

In the following, we intend to borrow some ideas in \cite{A10} to find normalized solutions of problem  \eqref{h2}.
		
		Let
		\begin{equation*}
			g_{\lambda } \left ( v \right ) =\frac{f \left (\phi^{-1} \left ( v \right ) \right )}{\varphi \left (\phi^{-1} \left ( v \right  )  \right )}-\lambda \frac{\phi^{-1} \left ( v \right ) }{\varphi \left (\phi^{-1} \left ( v \right )   \right )}+\lambda v ,
		\end{equation*}
		then  equation \eqref{h2} turns into
		\begin{equation}\label{scale}
			-\Delta v+\lambda v= g_{\lambda } \left ( v \right ),\, \quad x \in \mathbb{R}^{N}.
		\end{equation}

		\begin{lemma}\label{lemma2.1}
			Under the assumptions $\left(F_{1}\right)$-$\left(F_{3}\right)$ and $\left (\varphi _{0} \right )$ -$\left (\varphi _{1} \right )$, for any $\lambda>0$, there hold that
			\begin{itemize}
				\item [(i)] For some $1<p< 2^{*}-1$, $$\limsup_{s\rightarrow+\infty}\frac{g_\lambda(s)}{s^p}<\infty.$$
				
				\item[(ii)] $g_{\lambda } \left ( s \right )= o  \left ( s \right ) $, $s\to 0$.
				
				\item[(iii)] There exists $T > 0$ such that $\int_{0}^{T} g_{\lambda }\left ( \tau \right ) \mathrm{d} \tau> \frac{\lambda }{2} T^{2}$.

\item [(iv)]  $\lim \limits_{s \rightarrow+\infty} \frac{g_{\lambda } \left ( s \right )}{s^{\beta-1}}=\frac{\mu_2}{\left(a^*\right)^{\beta}}$.

\item [(v)] $g_\lambda(s)\le sg_\lambda'(s)$ for $s>0$ small.
			\end{itemize}
		\end{lemma}
		
		\begin{proof}
			\begin{itemize}
				\item [(i)] It is easy to check $(i)$ holds for $p=\beta-1$ by the assumption $(F_2)$.
				
				\item[(ii)] Since
				\begin{equation*}
					g_{\lambda } \left ( s \right ) =\frac{f \left (\phi^{-1} \left ( s \right ) \right )}{\varphi \left (\phi^{-1} \left ( s \right  )  \right )}-\lambda \frac{\phi^{-1} \left ( s \right ) }{\varphi \left (\phi^{-1} \left ( s \right )   \right )}+\lambda s ,
				\end{equation*}
then
				\begin{equation*}
					\begin{aligned}\lim_{s \to 0} \frac{g_{\lambda }\left ( s \right )  }{s}& =\lim_{s \to 0} \frac{f \left (\phi^{-1} \left ( s \right ) \right )}{s \varphi \left (\phi^{-1} \left ( s \right  )  \right )}-\lambda \lim_{s \to 0}\frac{\phi^{-1} \left ( s \right ) }{s \varphi \left (\phi^{-1} \left ( s \right )   \right )}+\lambda\\
						&=\lim_{s \to 0}\frac{f\left ( s \right ) }{s}=0.
					\end{aligned}
				\end{equation*}
				
				\item[(iii)] Since
				\begin{equation*}
					\begin{aligned}
						\lim_{t \to +\infty} \frac{\int_{0}^{t}g_{\lambda }\left ( \tau  \right )  \mathrm{d}\tau   }{t^{2} } &=\lim_{t \to +\infty}\frac{g_{\lambda }\left (t  \right ) }{2t} \\
						&= \frac{1}{2}\left [ \lim_{t \to +\infty}\frac{f \left (\phi^{-1} \left ( t \right ) \right )}{t \varphi \left (\phi^{-1} \left ( t \right  )  \right )}-\lambda \lim_{t \to +\infty}\frac{\phi^{-1} \left ( t \right ) }{t \varphi \left (\phi^{-1} \left ( t \right ) \right )}+\lambda \right ] .
					\end{aligned}
				\end{equation*}
Take $s=\phi ^{-1} \left ( t \right )$, then we get
				\begin{equation*}
					t=\phi \left ( s \right )=\int_{0}^{s}\varphi \left ( \tau  \right )  \mathrm{d}\tau
				\end{equation*}
				and
				\begin{equation*}
					\lambda \lim_{t \to +\infty}\frac{\phi^{-1} \left ( t \right ) }{t \varphi \left (\phi^{-1} \left ( t \right )   \right )}=\lambda \lim_{s \to +\infty}\frac{s}{\phi \left ( s \right ) \varphi \left ( s \right ) } ,
				\end{equation*}
Noting that $\lim\limits_{s \to +\infty} \varphi \left ( s \right ) =a^\ast$,
				\begin{equation*}
					\begin{aligned}
						\lambda \lim\limits_{s \to +\infty}\frac{s}{\phi \left ( s \right ) \varphi \left ( s \right ) }&=\frac{\lambda}{a^\ast}\lim\limits_{s \to +\infty}\frac{s}{\int_{0}^{s}\varphi \left ( \tau  \right )  \mathrm{d}\tau}\\
						&=\frac{\lambda}{a^\ast}\lim\limits_{s \to +\infty}\frac{1}{\varphi \left ( s \right )}\\
						&=\frac{\lambda }{(a^\ast)^{2}}.
					\end{aligned}
				\end{equation*}
Meanwhile, it follows from $(F_2)$ that
				\begin{equation*}
					\begin{aligned}
						\lim_{t \to +\infty}\frac{f \left (\phi^{-1} \left ( t \right ) \right )}{t \varphi \left (\phi^{-1} \left ( t \right  )  \right )}&=\lim_{s \to +\infty}\frac{f\left ( s \right ) }{\phi \left ( s \right ) \varphi \left ( s \right ) }=\frac{1}{a^\ast} \lim_{s \to +\infty} \frac{f\left ( s \right ) }{\phi \left ( s \right ) } \\
						&= \frac{1}{(a^\ast)^{2}} \lim_{s \to +\infty} \frac{f\left ( s \right ) }{s}=+\infty .
					\end{aligned}	   	
				\end{equation*}
Thus we obtain that
$$
\lim_{t \to +\infty} \frac{\int_{0}^{t}g_{\lambda }\left ( \tau  \right )  \mathrm{d}\tau   }{t^{2} } =\infty,
$$
which implies that there exists $T > 0$ such that $\int_{0}^{T} g_{\lambda }\left ( \tau \right ) \mathrm{d} \tau> \frac{\lambda }{2} T^{2}$.
\item [(iv)]
Set $p=\beta -1$ and notice that $\phi(s)\ge s$ for any $s>0$, we have
			\begin{equation*}\label{a2}
				\begin{aligned}
					\lim\limits  _{s \rightarrow+\infty} \frac{g_{\lambda}(s)}{s^{p}} & =\lim _{s \rightarrow+\infty} \frac{f\left(\phi^{-1}(s)\right.)}{s^{p} \varphi\left(\phi^{-1}(s)\right.)}+\lim _{s \rightarrow+\infty} \frac{\lambda s-\lambda \frac{\phi^{-1}(s)}{\varphi \left ( \phi ^{-1} \left ( s \right ) \right)}}{s^{p}} \\
					& =\lim \limits_{t \rightarrow+\infty} \frac{f(t)}{(\phi(t))^{p} \varphi(t)}-\lambda \lim _{t \rightarrow+\infty} \frac{t}{(\phi(t))^{p} \varphi(t)} \\
					& =\lim \limits_{t \rightarrow+\infty} \frac{f(t)}{(\phi(t))^{p} \varphi(t)}=\lim _{t \rightarrow+\infty} \frac{f(t)}{\left(ta^*\right)^{p} a^*} \\
					& =\lim \limits_{t \rightarrow+\infty} \frac{f(t)}{t^p} \cdot \frac{1}{\left(a^*\right)^{p+1}}=\frac{\mu_2}{\left(a^*\right)^{p+1}}.
				\end{aligned}
			\end{equation*}
\item [(v)] It suffices to show that $\varphi(t)g_\lambda(\phi(t))\le \phi(t)\frac{\rm{d}}{\rm{d}t}[g_\lambda(\phi(t))]$ for $t>0$ small. Obviously, 
\begin{align*}
\frac{\rm{d}}{\rm{d}t}[g_\lambda(\phi(t))]
&=\frac{f'(t)\varphi(t)-f(t)\varphi'(t)}{\varphi^2(t)}+\frac{\lambda\varphi^3(t)-\lambda\varphi(t)+\lambda t\varphi'(t)}{\varphi^2(t)}
\end{align*}
By $(\varphi_0)$, for any $t\ge0$, $\phi(t)\le t\varphi(t)$ and then
$$
\phi(t)\varphi^{-2}(t)\left[\varphi^3(t)-\varphi(t)+t\varphi'(t)\right]\ge\phi(t)\varphi(t)-t.
$$
Meanwhile, by $(F_2)$,
$$
\lim_{t\rightarrow0^+}\frac{\phi(t)\varphi(t)f'(t)}{f(t)}=\alpha-1>1,
$$
which implies that for $t>0$ small,
$$
\phi(t)\left[f'(t)\varphi(t)-f(t)\varphi'(t)\right]\ge\varphi^2(t)f(t).
$$
Thus, it follows that
$$
\phi(t)\frac{\rm{d}}{\rm{d}t}[g_\lambda(\phi(t))]\ge f(t)-\lambda t+\lambda \phi(t)\varphi(t)=\varphi(t)g_\lambda(\phi(t)),\,t>0\,\,\mbox{small}.
$$
			\end{itemize}
		\end{proof}
	
By Lemma \ref{lemma2.1}, for any fixed $\lambda>0$, $g_\lambda$ satisfies the Berestycki-Lions conditions. As a consequence of \cite[Theorem 1]{Lions}, for any $\lambda > 0$, problem \eqref{scale} admits a ground state solution in $H^1(\mathbb{R}^N)$, which is positive and radially symmetric.  Denoting
		\begin{equation*}
			\begin{aligned}
				\mathcal{S}= \Biggl\{&\left ( \lambda ,  \, v_{\lambda }  \right ) \in \left ( 0,  \, +\infty \right )\times H_{rad}^{1}\left ( \mathbb{R} ^{N}  \right ) : v_\lambda>0, \left ( \lambda ,  \, v_{\lambda }\right ) \, \text{solves} \\
				&-\Delta v+\lambda \frac{\phi^{-1} \left ( v \right ) }{\varphi \left (\phi^{-1} \left ( v \right )   \right )}-\frac{f \left (\phi^{-1} \left ( v \right ) \right )}{\varphi \left (\phi^{-1} \left ( v \right  )  \right )} =\,\,\mbox{in}\,\,\mathbb{R}^N     \Biggr\}.
			\end{aligned}	
		\end{equation*}
\br
As a consequence of \cite[Theorem 2]{GNN-1981}, for any $\lambda > 0$, for any positive $C^2$ solution $u$ of problem \eqref{scale} with $u(x)\rightarrow0$ as $|x|\rightarrow\infty$, one can show that $u$ is radially symmetric about some point $x_0\in \mathbb{R}^N$, that is, $u(x)=u_0(|x-x_0|)$, where $\frac{\partial u_0}{\partial r}<0$ for $r=|x-x_0|>0$. In what follows, we assume that $x_0=0$. Actually, thanks to the elliptic estimate, $u$ and $|\nabla u|$ decays exponentially at infinity. Set $g_\lambda(t)=h_1(t)+h_2(t), t\ge0$, where
$$
h_1(t) =\frac{f \left (\phi^{-1} \left ( t \right ) \right )}{\varphi \left (\phi^{-1} \left ( t\right  )  \right )},\,\,h_2(t)=-\lambda \frac{\phi^{-1} \left ( t \right ) }{\varphi \left (\phi^{-1} \left ( t \right )   \right )}+\lambda t ,
$$
then by $(F_2)$, $h_1(t)=O(|t|^{\alpha-1})$ as $t\rightarrow0$ and by $(\varphi_0)$,
\begin{align*}
\lim_{t\rightarrow0}\frac{h_2(t)}{t^2}&=\lambda\lim_{t\rightarrow0}\frac{\phi(t)\varphi(t)-t}{t^2}\\
&=\lambda\lim_{t\rightarrow0}\frac{\phi(t)[\varphi(t)-1]}{t^2}+\lambda\lim_{t\rightarrow0}\frac{\phi(t)-t}{t^2}\\
&=\frac{3\lambda}{2}\varphi'(0)=0.
\end{align*}
It follows that $g_\lambda(t)=O(|t|^{\min\{\alpha-1,2\}})$ as $t\rightarrow0$. Moreover,
$$
h_2'(t)=\lambda\varphi^{-3}(\phi^{-1}(t))\left[\varphi^3(\phi^{-1}(t))-\varphi(\phi^{-1}(t))+\phi^{-1}(t)\varphi'(\phi^{-1}(t))\right]\ge0,\,\,t\ge0,
$$
which implies that $h_2$ is nondecreasing for $t\ge0$. As for the condition on $h_1$ as follows, for some $C>0, p>1$,
$$
\left|h_1(t_1)-h_1(t_2)\right| \leq C|t_1-t_2| /|\log \min (t_1, t_2)|^p, \quad t_1,t_2\in[0,\max_{x\in\mathbb{R}^N}u(x)],
$$
to prove \cite[Theorem 2]{GNN-1981}, it is only used in the proof of \cite[Lemma 6.3]{GNN-1981}. Since $u$ decays exponentially, one can check that \cite[Lemma 6.3]{GNN-1981} still holds in our case and as well as \cite[Theorem 2]{GNN-1981}.
\er

\section{Asymptotic behaviors}
\renewcommand{\theequation}{3.\arabic{equation}}
In this section, we consider the asymptotic behaviors of $v_\lambda$ as $\lambda = \lambda_{n} \to 0^{+}$ or $\lambda = \lambda_{n} \to +\infty$. Similarly to \cite{A10}, we have the following result in the case of $\lambda = \lambda_{n} \to 0^{+}$.
		
		\begin{lemma}\label{lemma3.1}
		Let $\left\{v_{n}\right\}_{n=1}^{\infty} \subset \mathcal{S}$ with $\lambda=\lambda_{n} \rightarrow 0^{+}$, then the following asserts hold.
			\begin{itemize}
				\item [(i)]	
				\begin{equation*}\label{a3}
					\limsup\limits_{n \to +\infty}\left \| v_{n}  \right \| _{\infty } < + \infty.
				\end{equation*}
				
				\item [(ii)] 	
				\begin{equation*}\label{a4}
					\liminf _{n \rightarrow+\infty} \frac{\left\|v_{n}\right\|_{\infty}^{\alpha-2}}{\lambda_{n}}>0.
				\end{equation*}
				
				\item [(iii)] 	
				\begin{equation*}\label{a5}
					\limsup _{n \rightarrow+\infty} \frac{\left\|v_{n}\right\|_{\infty}^{\alpha-2}}{\lambda_{n}}<+\infty.
				\end{equation*}
				
				\item [(iv)] Set
				\begin{equation}\label{z7}
					w_{n}(x):=\lambda_{n}^{\frac{1}{2-\alpha}} v_{n}\left(\frac{x}{\sqrt{\lambda_{n}}}\right),
				\end{equation}
then $w_{n}$ satisfies
				\begin{equation*}\label{k1}
					-\Delta w_{n}+w_{n}=\frac{g_{\lambda _{n} } \left(\lambda_{n}^{\frac{1}{\alpha-2}} w_{n}\right)}{\lambda_{n}^{\frac{\alpha-1}{\alpha-2}}} \text { in } \mathbb{R}^{N}
				\end{equation*}	
and
				\begin{equation}\label{k2}
					w_{n} \rightarrow U \text{ in } H^{1}\left(\mathbb{R}^{N}\right)\,\mbox{and}\,\,C_{r, 0} \left ( \mathbb{R} ^{N}  \right ),
				\end{equation}
where $U \in C_{r,0}(\RN)$ is the unique positive solution of \eqref{k4}
\end{itemize}
		\end{lemma}
	
\bp
The proof is similar to \cite{A10} and we omit the details.
\ep
		In the following, we consider the case of $\lambda = \lambda_{n} \rightarrow +\infty$.
		
		\begin{lemma}\label{z8}
			Let $\left\{v_{n}\right\}_{n=1}^{\infty} \subset \mathcal{S}$ with $\lambda=\lambda_{n} \rightarrow+\infty$, then up to a subsequence,
			\begin{equation*}\label{o8}
				\liminf _{n \rightarrow+\infty}\left\|v_{n}\right\|_{\infty}=+\infty.
			\end{equation*}
		\end{lemma}
		
		\begin{proof}
			By regularity, for any fixed $n$, $v_{n} \in C^{2}\left ( \mathbb{R} ^{N}  \right )$ and we may suppose that $v_{n}\left ( 0 \right )=\left \| v_{n}  \right \| _{\infty }$.
			Setting
			\begin{equation*}
				\bar{{v_{n}}}\left ( x \right ) : = \frac{1}{v_{n} \left ( 0 \right ) } v_{n} \left ( \frac{x}{\sqrt{\lambda_{n} } } \right ),
			\end{equation*}
			we have
			\begin{equation}\label{t1}
				\begin{aligned}
					1=\bar{v_{n}}\left ( 0 \right ) \le& -\Delta \bar{v_{n}}\left ( 0 \right )+\bar{v_{n}}\left ( 0 \right )\\
					=&\frac{1}{\lambda_{n}v_{n}\left ( 0 \right ) } g_{\lambda_{n} } \left ( v_{n}\left ( 0 \right )  \right ) \\
					=&\frac{f\left ( \phi ^{-1}  \left ( v_{n}\left ( 0 \right ) \right ) \right ) }{\lambda_{n}v_{n}\left ( 0 \right )\varphi \left ( \phi ^{-1} \left ( v_{n}\left ( 0 \right )  \right ) \right ) }-\frac{ \phi ^{-1} \left ( v_{n}\left ( 0 \right )  \right) }{v_{n}\left ( 0 \right )\varphi \left ( \phi ^{-1} \left ( v_{n}\left ( 0 \right )  \right ) \right ) } +1\\
					\le& \frac{c\left ( \left | \phi ^{-1} \left ( v_{n}\left ( 0 \right )  \right) \right | ^{\alpha -1} +\left | \phi ^{-1} \left ( v_{n}\left ( 0 \right )  \right) \right | ^{\beta -1}  \right ) }{\lambda_{n}v_{n}\left ( 0 \right ) }  +1\\
					&-\frac{ \phi ^{-1} \left ( v_{n}\left ( 0 \right )  \right) }{v_{n}\left ( 0 \right )\varphi \left ( \phi ^{-1} \left ( v_{n}\left ( 0 \right )  \right ) \right ) } ,
				\end{aligned}
			\end{equation}
where we used the fact that for some $c>0$, $|f(t)|\le c(|t|^{\alpha-1}+|t|^{\beta-1}),\,t\ge0$.
			
Let $\phi ^{-1} \left ( v_{n}\left ( 0 \right )  \right) =a_{n} $, then $v_{n}\left ( 0 \right ) =\phi \left ( a_{n}  \right ) =\int_{0}^{a_{n}}\varphi \left ( \tau \right )\mathrm{d}\tau $. Next we argue by contradiction and suppose now that $\{a_{n}\}$ is bounded. Then (\ref{t1}) yields that
			\begin{equation}\label{k11}
				\begin{aligned}
					1 &\le \frac{c\left ( a_{n}^{\alpha -1} + a_{n}^{ \beta -1}\right ) }{\lambda_{n} \phi\left ( a_{n} \right ) }+1-\frac{a_{n}}{\phi(a_n)\varphi \left ( a_{n} \right ) } \\
					&\le \frac{c\left ( a_{n}^{\alpha -2} + a_{n}^{ \beta -2}\right ) }{\lambda_{n}}+1-\frac{a_{n}}{\phi(a_n)\varphi \left ( a_{n} \right ) } .
				\end{aligned}
			\end{equation}
If $a_{n} \to 0$, then $v_{n}\left(0\right)\to 0$. Recall that $v_{n}$ satisfies
			\begin{equation}\label{o10}
				-\Delta v_{n}+\lambda_{n}\frac{\phi ^{-1} \left ( v_{n} \right ) }{\varphi \left ( \phi ^{-1} \left ( v_{n} \right ) \right ) } =\frac{f\left ( \phi ^{-1} \left ( v_{n} \right ) \right ) }{\varphi \left ( \phi ^{-1} \left ( v_{n} \right ) \right ) }\,\,\mbox{in}\,\,\mathbb{R}^N.
			\end{equation}
Thanks to $(F_2)$, it follows that
			\begin{equation*}
				-\Delta v_{n}+\frac{\left ( \lambda_{n}-1 \right ) \phi ^{-1} \left ( v_{n} \right ) }{\varphi \left ( \phi ^{-1} \left ( v_{n} \right ) \right ) } \le 0 \,\,\mbox{in}\,\,\mathbb{R}^N.
			\end{equation*}
			Multiplying both sides by $v_{n}$ and integrating in $\mathbb{R}^N$, we get
			\begin{equation}\label{o11}
				\int_{\mathbb{R}^{N} }^{} \left | \nabla v_{n} \right | ^{2} \mathrm{d}x+\left ( \lambda_{n}-1 \right ) \int_{\mathbb{R}^{N} }^{} \frac{\phi ^{-1}\left ( v_{n}  \right ) v_{n}}{\varphi \left ( \phi ^{-1}\left ( v_{n}  \right ) \right )} \mathrm{d}x\le 0,
			\end{equation}
which is contradiction, due to $\phi ^{-1}\left ( v_{n}\left ( x \right ) \right )v_{n}\left ( x \right )>0, x\in\mathbb{R}^N$. So $\{a_{n}\}$ is positive and bounded away from zero. Passing to the limit of \eqref{k11} as $n\rightarrow\infty$, we have
$$
\limsup_{n\rightarrow\infty}\frac{a_{n}}{\phi(a_n)\varphi \left ( a_{n} \right ) }\le0,
$$
which is a contraction. Thus, up to a subsequence, $v_{n}\left ( 0 \right )  \to \infty $.
				\end{proof}
			
			\begin{lemma}\label{11}
					Let $\left\{v_{n}\right\}_{n=1}^{\infty} \subset \mathcal{S}$ with $\lambda=\lambda_{n} \rightarrow+\infty$, then
				\begin{equation}\label{o9}
					\liminf _{n \rightarrow+\infty} \frac{\left\|v_{n}\right\|_{\infty}^{\beta-2}}{\lambda_{n}}>0.
				\end{equation}
			\end{lemma}
			
			\begin{proof}
				Recall that
				\begin{equation*}
	           -\Delta v_{n}= \frac{f\left(\phi^{-1}\left( v_{n}\right)\right)}{\varphi\left(\phi^{-1}\left( v_{n}\right)\right)}-\lambda_{n} \frac{\phi^{-1}\left( v_{n}\right)}{ \varphi \left(\phi^{-1}\left( v_{n}\right)\right)}\,\,\mbox{in}\,\,\mathbb{R}^N.
				\end{equation*}
Without loss of generality, we suppose that $v_{n} \left(0\right)=\left\|v_{n}\right\|_{\infty}, \, \forall n \in \mathbb{N}$. Since $-\Delta v_{n}(0) \ge 0$,
				\begin{equation*}
					f\left(\phi^{-1}\left(\left\|v_{n}\right\|_{\infty}\right)\right) \ge \lambda_{n} \phi^{-1}\left( \left\|v_{n}\right\|_{\infty}\right) ,
				\end{equation*}
which yields that
				\begin{equation*}
					\liminf _{n \rightarrow +\infty}\frac{\left\|v_{n}\right\|_{\infty}^{\beta-{2}}}{\lambda_{n}}>0,
				\end{equation*}
			so $(\ref{o9})$ is proved.
			\end{proof}
		
		\begin{lemma}\label{12}
			Let $\left\{v_{n}\right\}_{n=1}^{\infty} \subset \mathcal{S}$ with $\lambda=\lambda_{n} \rightarrow+\infty$, then
			\begin{equation}
				\limsup\limits_{n \to +\infty}\frac{\left \| v_{n}  \right \| _{\infty }^{\beta -2}  }{\lambda _{n} }< +\infty .
			\end{equation}
		\end{lemma}
		
		\begin{proof}
			We prove it by contradiction and suppose
			\begin{equation*}
				\lim\limits_{n \to +\infty}\frac{\left\|v_{n}\right\|_{\infty}^{\beta-2}}{\lambda_{n}}=+\infty .
			\end{equation*}
Set $k=\frac{\beta-2 }{2}$ and
			\begin{equation*}
				 \tilde{{v}_{n}} (x)=\frac{1}{\left\|v_{n}\right\|_{\infty}} v_{n}\left(\frac{x}{\left\|v_{n}\right\|_{\infty}^{k}}\right),
			\end{equation*}
then $\tilde{{v}_{n}} \left(0\right)=\left\|\tilde{{v}_{n}}\right\|_{\infty}=1$ and
			\begin{equation*}
				-\Delta \tilde{{v}_{n}}(x)=-\frac{1}{\left\|v_{n}\right\|_{\infty}^{1+2 k}} \Delta v_{n}\left(\frac{x}{\left\|v_{n}\right\|_{\infty}^{k}}\right),
			\end{equation*}
				then
				\begin{equation}\label{o101}
					-\Delta \tilde{{v}_{n}}(x)=\frac{1}{\left\|v_{n}\right\|_{\infty}^{1+2 k}}\left[\frac{f\left(\phi^{-1}\left(\tilde{{v}_{n}}(x)\left\|v_{n}\right\|_{\infty}\right)\right)}{\varphi\left(\phi^{-1}\left(\tilde{{v}_{n}}(x)\left\|v_{n}\right\|_{\infty}\right)\right)}-\lambda_{n} \frac{\phi^{-1}\left(\tilde{{v}_{n}}(x)\left\|v_{n}\right\|_{\infty}\right)}{\varphi\left(\phi^{-1}\left(\tilde{{v}_{n}}(x)\left\|v_{n}\right\|_{\infty}\right)\right)}\right],
				\end{equation}
By $(F_2)$, the right hand side of (\ref{o101}) is in $L^{\infty} \left(\mathbb{R}^{N}\right)$. Up to a subsequence, we may suppose that $\tilde{v_{n} }  \to \tilde{v} $ in $C_{loc}^{2} \left ( \mathbb{R} ^{N}  \right )$ and then $\tilde{v}$  is a non-negative bounded solution to
		 \begin{equation*}
		 	-\Delta \tilde{v} =\mu_2(a^\ast)^{-\beta}\tilde{v}^{\beta -1}  \text { in } \mathbb{R}^{N}.
		 \end{equation*}
As a consequence of \cite[Theorem 2.2-(ii)]{A10}, $\tilde{v}\equiv0$, which contradicts $\tilde{v}\left(0\right)= 1$.
		\end{proof}
		
		\begin{lemma}\label{s1}
			Let $\left\{v_{n}\right\}_{n=1}^{\infty} \subset \mathcal{S}$ with $\lambda=\lambda_{n} \rightarrow+\infty$. Define
			\begin{equation}\label{o12}
				w_{n}(x):=\lambda_{n}^{\frac{1}{2-\beta}} v_{n}\left(\frac{x}{\sqrt{\lambda_{n}}}\right),
			\end{equation}
then, up to a subsequence, $w_{n} \to V^\ast$ in $C_{r, 0} \left ( \mathbb{R} ^{N}  \right ) $ and $H^{1}\left(\mathbb{R}^{N}\right)$ as $n \to +\infty$,
where $V^\ast=a^\ast V(\frac{\cdot}{a^\ast})$ and $V$ is the unique positive solution of \eqref{k5}.

		\begin{proof}
			By Lemma \ref{11} and Lemma \ref{12}, we have that
			\begin{equation*}
					0<\liminf _{n \rightarrow+\infty} \frac{v_{n}\left(0\right)^{\beta-2}}{\lambda_{n}} \le	\limsup\limits_{n \to +\infty}\frac{v_{n}\left(0\right)^{\beta -2}  }{\lambda _{n} }< +\infty,
			\end{equation*}
			which implies $\left \{ w_{n} \right \}$ is uniformly bounded in $L^{\infty} \left(\mathbb{R}^{N}\right)$.
Since $w_{n}$ satisfies
			\begin{equation}\label{z11}
				-\Delta w_{n}+w_{n}=\frac{g_{\lambda_{n}}\left(\lambda_{n}^{\frac{1}{\beta-2}} w_{n}\right)}{\lambda_{n}^{\frac{\beta-1}{\beta-2}}} \text { in } \mathbb{R}^{N},
			\end{equation}
that is,
		\begin{equation}\label{o13}
		-\Delta w_{n}=\lambda_{n}^{\frac{1-\beta}{\beta-2}} \frac{f\left(\phi^{-1}\left(\lambda_{n} ^\frac{1}{\beta^{-2}} w_{n}\right)\right)}{\varphi\left(\phi^{-1}\left(\lambda_{n}^{\frac{1}{\beta^{-2}}} w_{n}\right)\right)}-\lambda_{n}^{\frac{1}{2-\beta}} \frac{\phi^{-1}\left(\lambda_{n}^{\frac{1}{\beta^{-2}}} w_{n}\right)}{\varphi\left(\phi^{-1}\left(\lambda_{n}^{\frac{1}{\beta-2}} w_{n}\right)\right)} .
		\end{equation}
%
Thanks to $(F_2)$, one can check that the right hand of (\ref{o13}) is in $L^{\infty}\left(\mathbb{R}^{N}\right)$. By the elliptic regularity, up to a subsequence,
 we assume that $w_{n} \rightarrow V^\ast$  in  $C_{ loc }^{2}\left(\mathbb{R}^{N}\right) $, where $V^\ast$ satisfies
 $$
-\Delta V^\ast+(a^\ast)^{-2}V^\ast=\mu_2 (a^\ast)^{-\beta}(V^\ast)^{\beta-1}~\hbox{in}~\RN,\,\,V^\ast(0)=\max_{x\in\mathbb{R}^N}V^\ast(x).
$$	
Similarly to \cite{A10}, the result desired can be obtained.
		\end{proof}
		\end{lemma}
		
		
Next, we focus on the asymptotic behaviour of
$\left\|\phi ^{-1} \left ( v_{n}\right )\right\|_{2} $ as $n\rightarrow\infty$.

\begin{theorem}\lab{Th2}
		\begin{itemize}
		 \item[(i)] Let $\left\{ v_{n}  \right\}_{n=1}^{\infty} \subset H^{1}\left(\mathbb{R}^{N}\right)$ be positive solutions to (\ref{h2}) with $\lambda=\lambda_{n} \rightarrow 0^{+}$. Then
		 \begin{equation*}
		 	\lim _{n \rightarrow +\infty}\left\|v_{n}\right\|_{\infty}=0, \quad \lim _{n \rightarrow +\infty}\left\|\nabla v_{n}\right\|_{2}=0
		 \end{equation*}
		 and
		 \begin{equation*}
		 	\lim _{n \rightarrow +\infty}\left\|\phi^{-1} \left ( v_{n} \right )\right\|_{2}= \begin{cases}0 & \alpha<2+\frac{4}{N}, \\
		 		\|U\|_{2} & \alpha=2+\frac{4}{N} ,\\
		 		+\infty & \alpha>2+\frac{4}{N}.\end{cases}
		 \end{equation*}
	
	  \item[(ii)]  Let $\left\{ v_{n}  \right\}_{n=1}^{\infty} \subset H^{1}\left(\mathbb{R}^{N}\right)$  be positive solutions to (\ref{h2}) with $\lambda=\lambda_{n} \rightarrow+\infty$. Then
	
	  \begin{equation*}
	  	\lim _{n \rightarrow +\infty}\left\| v_{n} \right\|_{\infty}=+\infty, \quad \lim _{n \rightarrow +\infty}\left\|\nabla v_{n} \right\|_{2}=+\infty
	  \end{equation*}
	  and
	  \begin{equation*}
	  	\lim _{n \rightarrow +\infty}\left\|\phi^{-1} \left ( v_{n} \right )\right\|_{2}= \begin{cases}+\infty & \beta<2+\frac{4}{N}, \\
	  		(a^\ast)^{N/2}\|V\|_{2} & \beta=2+\frac{4}{N},
	  		\\ 0 & \beta>2+\frac{4}{N}.\end{cases}
	  \end{equation*}
		\end{itemize}
		\end{theorem}
\bp Similarly to \cite{A10}, one can show the asymptotic behaviour of $\| v_{n} \|_{\infty}$ and $\| \nabla v_{n} \|_{2}$. In the following, we only consider
$\left\|\phi ^{-1} \left ( v_{n}\right )\right\|_{2}^{2} $.

		\begin{itemize}
			\item [(i)]
			 Let $w_{n}$ be defined by (\ref{z7}) and by Lemma \ref{lemma3.1}, we have
			\begin{equation*}
				\begin{aligned}
					\left \| \phi^{-1}\left ( v_{n}  \right )\right \|_{2}^{2}&=\int_{\mathbb{R}^{N}}^{} \left | \phi ^{-1} \left ( v_{n}\left ( x \right )   \right )  \right | ^{2}\mathrm{d}x\\
					&=\lambda _{n}^{-\frac{N}{2} } \int_{\mathbb{R}^{N}}^{} \left | \phi ^{-1} \left ( v_{n}\left ( \frac{y}{\sqrt{\lambda _{n} } }  \right ) \right ) \right | ^{2}\mathrm{d}y\\
					&=\lambda _{n}^{-\frac{N}{2} } \int_{\mathbb{R}^{N}}^{} \left ( \frac{\phi ^{-1}\left ( v_{n}\left ( \frac{y}{\sqrt{\lambda _{n} } } \right )  \right )  }{v_{n}\left ( \frac{y}{\sqrt{\lambda _{n} } } \right )  }  \right ) ^{2}\left ( v_{n} \left (  \frac{y}{\sqrt{\lambda _{n} } } \right )  \right )  ^{2}\mathrm{d}y \\
					&= \lambda _{n}^{-\frac{N}{2} }  \int_{\mathbb{R}^{N}}^{} \left ( \frac{\phi ^{-1}\left ( v_{n}\left ( \frac{y}{\sqrt{\lambda _{n} } } \right )  \right )  }{v_{n}\left ( \frac{y}{\sqrt{\lambda _{n} } } \right )  }  \right ) ^{2}\left ( \frac{w_{n}\left ( y \right )}{\lambda_{n}^{\frac{1}{2-\alpha } } }  \right )  ^{2}\mathrm{d}y \\
					&= \lambda _{n}^{-\frac{N}{2}+\frac{2}{\alpha-2} } \int_{\mathbb{R}^{N}}^{}\left ( \frac{\phi ^{-1}\left ( v_{n}\left ( \frac{y}{\sqrt{\lambda _{n} } } \right )\right )  }{v_{n}\left ( \frac{y}{\sqrt{\lambda _{n} } } \right )}  \right )^{2} w_{n}^{2}\left ( y \right ) \mathrm{d}y\\
					&= \lambda _{n}^{-\frac{N}{2}+\frac{2}{\alpha-2} } (\left \| U  \right \| _{2}^{2}+o_n(1)),
				\end{aligned}
			\end{equation*}
which yields the result desired.
			
			\item [(ii)]
			For the case of $\lambda = \lambda_{n} \to +\infty$, let $w_{n}$ be defined by (\ref{o12}). Since
			\begin{equation*}
				\lim_{s \to +\infty} \frac{\phi ^{-1} \left ( s \right ) }{s} =\lim_{t \to +\infty} \frac{t}{\phi \left ( t \right ) }=\frac{1}{a^\ast}  ,
			\end{equation*}
by Lemma \ref{s1},
			\begin{equation*}
				\begin{aligned}
					\left \| \phi^{-1}\left ( v_{n}  \right )\right \|_{2}^{2}
					=& \lambda _{n}^{-\frac{N}{2} } \int_{\mathbb{R}^{N}}^{} \left ( \frac{\phi ^{-1}\left ( v_{n}\left ( \frac{y}{\sqrt{\lambda _{n} } } \right )  \right )  }{v_{n}\left ( \frac{y}{\sqrt{\lambda _{n} } } \right )  }  \right ) ^{2}\left ( v_{n} \left (  \frac{y}{\sqrt{\lambda _{n} } } \right )  \right )  ^{2}\mathrm{d}y \\
					&= \lambda _{n}^{-\frac{N}{2}+\frac{2}{\beta-2} } \int_{\mathbb{R}^{N}}^{}\left ( \frac{\phi ^{-1}\left ( v_{n}\left ( \frac{y}{\sqrt{\lambda _{n} } } \right )\right )  }{v_{n}\left ( \frac{y}{\sqrt{\lambda _{n} } } \right )}  \right )^{2} w_{n}^{2}\left ( y \right ) \mathrm{d}y\\
					&= \lambda _{n}^{-\frac{N}{2}+\frac{2}{\beta-2} } ((a^\ast)^{-2}\left \| V^\ast  \right \| _{2}^{2}+o_n(1)),
				\end{aligned}
			\end{equation*}
which yields the result desired. 	
		\end{itemize}			
\ep

\s{Uniqueness for $\lambda$ small or large}
		\renewcommand{\theequation}{4.\arabic{equation}}
		\hspace*{0.5cm}
		In this section, we prove the uniqueness of positive solutions to (\ref{h2}) provided $\lambda> 0$ small or large enough.
		\begin{theorem}\label{unique}
			 Let $N \geq 3$ and assume that $\left(\varphi_{0}\right)$-$\left(\varphi_{1}\right)$, $\left(F_{1}\right)$-$\left(F_{2}\right)$ hold.
Then (\ref{h2}) has at most one positive solution if $\lambda>0$ large and in addition $\left(\varphi_{2}\right)$ holds or if $\lambda>0$ small and in addition $\left(F_{3}\right)$ holds.
		\end{theorem}
		
		\begin{proof}
		We adopt the idea in \cite{A10} to give the proof.

If $\lambda>0$ small and $\left(F_{1}\right)$-$\left(F_{3}\right)$ hold, assume that problem (\ref{h2}) admits two families of positive solutions
$u_{\lambda}^{(1)}$ and $u_{\lambda}^{(2)}$ with $\lambda\rightarrow 0^+$. Let
$$
  v_{\lambda}^{(i)}(\cdot)
	 :=\lambda^{-\frac{1}{\alpha-2}}u_{\lambda}^{(i)}
	       \left(\cdot/\sqrt{\lambda}\right),\quad i=1,2.
$$
Then by Lemma \ref{lemma3.1},
$v_{\lambda}^{(1)},
 v_{\lambda}^{(2)}\in H_{rad}^{1}(\RN)$
satisfy
$$-\Delta v+v=\frac{g_\lambda(\lambda^{\frac{1}{\alpha-2}}v)}{\lambda^{\frac{\alpha-1}{\alpha-2}}}\,\mbox{in}\,\,\mathbb{R}^N $$
and as $\lambda\rightarrow 0^+$,
$v_{\lambda}^{(i)}\rightarrow U\;\hbox{ both in}\;C_{r,0}(\RN)~\hbox{and in}~H^1(\RN), \;i=1,2.$
Setting
$$
\xi_\lambda:=\frac{v_{\lambda}^{(1)}-v_{\lambda}^{(2)}}{\left\|v_{\lambda}^{(1)}-v_{\lambda}^{(2)}\right\|_{\infty}},
$$
for any $x \in \RN$, there exists some $\theta(x)\in [0,1]$ such that
\begin{equation*}
-\Delta \xi_\lambda +\xi_\lambda=\lambda^{-1}g_\lambda'\left(\lambda^{\frac{1}{\alpha-2}}[\theta(x)v_{\lambda}^{(1)}(x)
+(1-\theta(x))v_{\lambda}^{(2)}(x)]\right)  \xi_\lambda\,\,\,\mbox{in}\,\,\mathbb{R}^N ,
\end{equation*}
where for any $t\in\mathbb{R}$,
\begin{align*}
\lambda^{-1}g_{\lambda}'(t)=&\frac{f'(\phi^{-1}(t))\varphi(\phi^{-1}(t))-f(\phi^{-1}(t))\varphi'(\phi^{-1}(t))}{\lambda[\varphi(\phi^{-1}(t))]^3}\\
&-\frac{\varphi(\phi^{-1}(t))-\phi^{-1}(t)\varphi'(\phi^{-1}(t))}{[\varphi(\phi^{-1}(t))]^3}+1.
\end{align*}
Thanks to $v_{\lambda}^{(i)}\rightarrow U$ in $C_{r,0}(\RN)$, one can see that
$$
\lim_{\lambda\rightarrow0^+}\lambda^{\frac{1}{\alpha-2}}[\theta(x)v_{\lambda}^{(1)}(x)
+(1-\theta(x))v_{\lambda}^{(2)}(x)]=0\,\,\,\mbox{uniformly for}\,\,x\in\mathbb{R}^N,
$$
which implies that for any $x \in \RN$,
$$
\lim_{\lambda\rightarrow0^+}\lambda^{-1}g_\lambda'\left(\lambda^{\frac{1}{\alpha-2}}[\theta(x)v_{\lambda}^{(1)}(x)
+(1-\theta(x))v_{\lambda}^{(2)}(x)]\right)=\mu_1(\alpha-1)U^{\alpha-2}.
$$
Hence, up to a subsequence, $\xi_\lambda\rightarrow \xi$ in $C_{loc}^{2}(\RN)$ and $\xi$ is a radial bounded solution of
$$
-\Delta \xi+\xi=(\alpha-1)\mu_1 U^{\alpha-2}\xi.
$$
Similarly to \cite{A10}, we can get one contradiction.
\vskip0.1in
Now, we consider the case of $\lambda >0$ large and assume $\left(\varphi_{2}\right)$ holds. Assume problem (\ref{h2}) admits two families of positive solutions
 $u_{\lambda}^{(1)}$ and $u_{\lambda}^{(2)}$ with $\lambda\rightarrow +\infty$. Let
$$
  v_{\lambda}^{(i)}(\cdot)
	 :=\lambda^{-\frac{1}{\beta-2}}u_{\lambda}^{(i)}
	       \left(\cdot/\sqrt{\lambda}\right),\quad i=1,2,
$$
and
$$
\xi_\lambda:=\frac{v_{\lambda}^{(1)}-v_{\lambda}^{(2)}}{\left\|v_{\lambda}^{(1)}-v_{\lambda}^{(2)}\right\|_{\infty}}.
$$
Similarly as above, there exists some $\theta(x)\in [0,1]$ such that
\begin{equation*}
-\Delta \xi_\lambda +\xi_\lambda =\lambda^{-1}g'\left(\lambda^{\frac{1}{\beta-2}}[\theta(x)v_{\lambda}^{(1)}(x)+(1-\theta(x))v_{\lambda}^{(2)}(x)]\right) \xi_\lambda~\hbox{in}~\RN.
\end{equation*}
By Lemma \ref{s1}, $v_{\lambda}^{(i)}\rightarrow V$ in $C_{r,0}(\RN)$ as $\lambda\rightarrow+\infty$ for $i=1,2$. It is easy to know that, for any $x\in\mathbb{R}^N$,
$$
\lim_{\lambda\rightarrow+\infty}\lambda^{\frac{1}{\beta-2}}[\theta(x)v_{\lambda}^{(1)}(x)
+(1-\theta(x))v_{\lambda}^{(2)}(x)]=+\infty.
$$
By $\left(\varphi_{2}\right)$, for any $x\in\mathbb{R}^N$,
$$
\lim_{\lambda\rightarrow+\infty}\left.\frac{\varphi(\phi^{-1}(t))-\phi^{-1}(t)\varphi'(\phi^{-1}(t))}{[\varphi(\phi^{-1}(t))]^3}\right|_{t=\lambda^{\frac{1}{\beta-2}}[\theta(x)v_{\lambda}^{(1)}(x)
+(1-\theta(x))v_{\lambda}^{(2)}(x)]}=(a^\ast)^{-2}
$$
and
$$
\lim_{\lambda\rightarrow+\infty}\left.\frac{f'(\phi^{-1}(t))\varphi(\phi^{-1}(t))-f(\phi^{-1}(t))\varphi'(\phi^{-1}(t))}{\lambda[\varphi(\phi^{-1}(t))]^3}\right|_{t=\lambda^{\frac{1}{\beta-2}}[\theta(x)v_{\lambda}^{(1)}(x)
+(1-\theta(x))v_{\lambda}^{(2)}(x)]}=\mu_2 (a^\ast)^{-\beta}V^{\beta-2}.
$$
Hence, up to a subsequence, $\xi_\lambda \rightarrow \xi$ in $C_{loc}^{2}(\RN)$ and $\xi$ is a radial bounded solution of
$$
-\Delta \xi+(a^\ast)^{-2}\xi=\mu_2 (a^\ast)^{-\beta}V^{\beta-2}\xi.
$$
Similarly, we get a contradiction.
		\end{proof}

		\vskip0.1in


		\s{Proof of Theorem \ref{Th1}}
		\renewcommand{\theequation}{5.\arabic{equation}}
		\hspace*{0.5cm}
		\noindent
We give the proof of Theorem \ref{Th1}. Being similar to that of \cite{A10}, we only sketch it.
\smallskip

$\bullet$ {\bf Step 1.} For any $\lambda>0$, problem (\ref{h2}) admits a positive and radially symmetric solution $u_\lambda$. By Theorem \ref{unique}, there exists $\lambda_0>0$ such that, up to
translation, $u_\lambda$ is the unique positive solution of problem (\ref{h2}) and then is of mountain-pass-type (see \cite[Definition 2.1]{A10}). Furthermore,  the map $\lambda\mapsto u_\lambda, \lambda\in (0,\lambda_0)$
is continuous. That is, $\left\{(\lambda, u_\lambda):\lambda\in (0,\lambda_0)\right\}$ is a curve in $\mathbb{R}\times H_{rad}^{1}(\RN)$.
\vskip0.1in
$\bullet$  {\bf Step 2.} Let $\widetilde{\mathcal{S}} \subset \mathcal{S}$ be the connected component of $\mathcal{S}$ containing the curve
$\left\{(\lambda, u_\lambda):\lambda\in (0,\lambda_0)\right\}$, where $\mathcal{S}$ is given in Section 2. Denoting by
$P_{1}:(0,+\infty) \times H_{\text {rad }}^{1}\left(\mathbb{R}^{N}\right) \rightarrow(0,+\infty)$ the projection onto the $\lambda$-component,
we show that $P_{1}(\widetilde{\mathcal{S}})=(0,+\infty)$.

For any $v\in H_{\text {rad }}^{1}(\mathbb{R}^{N})$, define
		$$
		\mathbb{T}_{\lambda}(v)=:(-\Delta+\lambda)^{-1} g_\lambda(v).
		$$
Then $\mathbb{T}_\lambda: H_{rad}^{1}(\RN)\rightarrow H_{rad}^{1}(\RN)$ is completely continuous and $u$ is a radial solution of problem (\ref{h2})
if and only if $u$ is a fixed point of $\mathbb{T}_{\lambda}$ in $H_{\text {rad }}^{1}(\mathbb{R}^{N})$. For $\lambda\in (0,\lambda_0)$, thanks to Lemma \ref{lemma2.1}-$(v)$, similarly to \cite[Lemma 7.2]{A10}, one can get the local fixed point index
$$
{\rm ind}(\mathbb{T}_\lambda, u_\lambda)={\rm deg}_{LS}\left({\rm Id}-\mathbb{T}_\lambda, N_\varepsilon(u_\lambda), 0\right)=-1,
$$
where $\varepsilon>0$ is small, $N_\varepsilon$ denotes the $\varepsilon$-neighborhood in $H^1_{rad}(\RN)$.
		
\vskip0.1in
$\bullet$ {\bf Step 3.} For any fixed $0<a<b$, define
\begin{equation*}
			\begin{aligned}
				\mathcal{S}(a,b)= \Biggl\{&\left ( \lambda ,  \, v_{\lambda }  \right ) \in \left [ a,  \, b \right ]\times H_{rad}^{1}\left ( \mathbb{R} ^{N}  \right ) : v_\lambda>0, \left ( \lambda ,  \, v_{\lambda }\right ) \, \text{solves} \\
				&-\Delta v+\lambda \frac{\phi^{-1} \left ( v \right ) }{\varphi \left (\phi^{-1} \left ( v \right )   \right )}-\frac{f \left (\phi^{-1} \left ( v \right ) \right )}{\varphi \left (\phi^{-1} \left ( v \right  )  \right )} =\,\,\mbox{in}\,\,\mathbb{R}^N     \Biggr\}.
			\end{aligned}	
		\end{equation*}		
By using the similar blow-up technique and ODE approach, one can show that the set $\mathcal{S}(a,b)$ is compact in $C_{r,0}(\RN)$ and $H_{rad}^1(\RN)$. Via the topological degree theory, thanks to the compactness of $\mathcal{S}(a,b)$ and ${\rm ind}(\mathbb{T}_\lambda, u_\lambda)=-1$ for $\lambda\in (0,\lambda_0)$,
there holds that $P_{1}(\widetilde{\mathcal{S}})=(0, \infty)$.

\vskip0.1in
$\bullet$ {\bf Step 4.} Similarly to \cite{A10}, define
			$$
				\tilde \rho: \tilde{\mathcal{S}} \rightarrow \mathbb{R}^{+}, \quad(\lambda, v) \mapsto\|\phi^{-1} \left ( v \right )\|_{2}^{2} .
			$$
We just consider the case $(i)$. By $P_{1}(\widetilde{\mathcal{S}})=(0,+\infty)$ and Theorem \ref{Th2}, if $2<\alpha<2+\frac{4}{N}$, then there exists
$\left(\lambda_{n}, \, v_{n}\right) \subset \widetilde{\mathcal{S}} $ with $\lambda_{n} \rightarrow 0^{+}$ and $\left\|\phi^{-1} \left ( v_{n} \right )\right\|_{2}^{2} \rightarrow 0$.
Similarly, if $2<\beta<2+\frac{4}{N}$, there exists $\left(\lambda_{n}^{\prime},  \, v_{n}^{\prime}\right) \subset \widetilde{\mathcal{S}}$
with $\lambda_{n}' \rightarrow+\infty$ and $\|\phi^{-1} \left ( v_{n}' \right )\|_{2}^{2} \rightarrow+\infty$. Since $\widetilde{\mathcal{S}}$ is connected,
for any given $c>0$, there exists
$(\lambda_c,u_c)\in \widetilde{\mathcal{S}}$ such that $\tilde \rho(u_c)=c$, that is, (\ref{h2}) possesses a positive normalized solution. Since the proof of the other cases can be done similarly as in
\cite{A10}, we omit the details.

		\vskip 5mm
		\noindent\textbf{Conflict of interest.} The authors have no competing interests to declare that are relevant to the content of this article
		
%


\medskip
		
		
		
		

\begin{thebibliography}{100}
\bibitem{A1}
C. O. Alves, Y. Wang, Y. Shen, Soliton solutions for a class of quasilinear Schr\"odinger equations with a parameter, {\it J. Differ. Equ.}, {\bf259} (2015), 318-343.

\bibitem{Lions} H. Berestycki and P. L. Lions, Nonlinear scalar field equations I. Existence of a ground state, {\it Arch. Ration. Mech. Anal.,} {\bf 82}(1983), 313-346.
		
%

\bibitem{byeon} J. Byeon, L. Jeanjean, Standing waves for nonlinear Schr\"{o}dinger equations with a general nonlinearity , {\it Arch. Ration. Mech. Anal., } {\bf 185}(2007),  185-200.	
\bibitem{A4}
M. Colin, L. Jeanjean, Solutions for quasilinear Schr\"odinger equation: a dual approach, {\it Nonlinear Anal.}, {\bf56}(2004), 213-226.
	
\bibitem{CJS} M. Colin, L. Jeanjean, M. Squassina, Stability and instability results for standing waves of quasilinear Schr\"odinger equations, {\it Nonlinearity}, {\bf23} (2010), 1353-1385.
	
\bibitem{A5}
Y. Deng, S. Peng, S. Yan, Critical exponents and solitary wave solutions for generalized quasilinear Schr\"odinger equations, {\it J. Differ. Equ.}, {\bf260}(2016), 1228-1262.

\bibitem{GNN-1981} B. Gidas, W.M.~Ni, L.~Nirenberg, Symmetry of positive solutions of nonlinear elliptic equations in {${\mathbb{R}}^{n}$}. In {\em Mathematical analysis and applications, {P}art {A}}, volume~7 of {\em Adv. in Math. Suppl. Stud.}, pp. 369--402. Academic Press, New York-London,1981.
		
\bibitem{A6}
E. Gloss, Existence and concentration of positive solutions for a quasilinear equation in $\mathbb{R}^{N}$, {\it J. Math. Anal. Appl.}, {\bf371} (2010), 465-484.
		
		
\bibitem{A8}
R. Hasse, A general method for the solution of nonlinear soliton and kink Schr\"odinger equations, {\it  Z. Physik B}, {\bf37} (1980), 83-87.
	
\bibitem{hzz} C. Huang, J. Zhang, X. Zhong, Existence and multiplicity of solutions for general quasi-linear elliptic equations with sub-cubic nonlinearities, {\it J. Math.Anal.Appl.}, {\bf531}(2024), 127880.

\bibitem{Jeanjean1997}
L. Jeanjean, Existence of solutions with prescribed norm for semilinear elliptic equations, {\em Nonlinear Anal.}, {\bf28}(1997), no. 10, 1633--1659.	

\bibitem{A9}
L. Jeanjean, T. Luo, Z.-Q. Wang, Multiple normalized solutions for quasilinear Schr\"odinger equations, {\it J. Differ. Equ.}, {\bf259}(2015), 3894-3928.
		
\bibitem{A10}
L. Jeanjean, J. J. Zhang, X. X. Zhong, A global branch approach to normalized solutions for the Schr\"odinger equation, {\it J. Math. Pures Appl.}, {\bf183}(2024), 44-75.
		
\bibitem{A11}
S. Kurihura, Large-amplitude quasi-solitons in superfluid films, {\it J. Phys. Soc. Japan}, {\bf50} (1981), 3262-3267.
		
\bibitem{A12}
A. Kosevich, B. Ivanov, A. Kovalev, Magnetic solitons, {\it Phys. Rep.}, {\bf194} (1990), 117-238.

\bibitem{LSS}
E. W. Laedke, K. H. Spatschek, L. Stenflo, Evolution theorem for a class of perturbed envelope soliton solutions, {\it J. Math. Phys.}, {\bf24} (1983), 2764-2769.
	
\bibitem{A13}
A. Litvak, A. Sergeev, One dimensional collapse of plasma waves, {\it JETP Lett.}, {\bf27} (1978) 517-520.
		
\bibitem{A14}
J. Liu, Z.-Q. Wang, Soliton solutions for quasilinear Schr\"odinger equations I, {\it Proc. Amer. Math. Soc.}, {\bf131}(2002), 441-448.
		
\bibitem{A15}
J. Liu, Z.-Q. Wang, Soliton solutions for quasilinear Schr\"odinger equations II, {\it J. Differ. Equ. }, {\bf187} (2003), 473-493.
		
\bibitem{A16}
H. Li, W. Zou, Quasilinear Schr\"odinger equations: ground state and infinitely many normalized solutions, {\it Pacific J. Math.}, {\bf322} (2023), 99-138.

\bibitem{LiuX1} X. Liu, J. Liu, Z.-Q. Wang, Quasilinear elliptic equations via perturbation method, {\it Proc. Am. Math. Soc.}, {\bf141} (2013), 253-263.
	
\bibitem{LiuX3} X. Liu, J. Liu, Z.-Q. Wang, Quasilinear elliptic equations with critical growth via perturbation method, {\it J. Differential Equations}, {\bf254} (2013), 102-124.
	
\bibitem{A17}
V. Makhankov, V. Fedyanin, Nonlinear effects in quasi-one-dimensional models of condensed matter theory, {\it Phys. Rep.}, {\bf104} (1984), 1-86.
				
\bibitem{pellacci}
B. Pellacci, M. Squassina,
Unbounded critical points for a class of lower semicontinuous functionals
{\it J. Differential Equations} {\bf  201} (2004), 25-62.		
		
\bibitem{A18}
G. Quispel, H. Capel, Equation of motion for the Heisenberg spin chain, {\it Phys. A }, {\bf110} (1982), 41-80.
		
\bibitem{A19}
Y. Shen, Y. Wang, Soliton solutions for generalized quasilinear Schr\"odinger equations, {\it Nonlinear Anal.}, {\bf80} (2013), 194-201.

\bibitem{Stuart1981}
C. A. Stuart, Bifurcation from the continuous spectrum in the {$L^2$}-theory of elliptic equations on {${\bf R}^n$}. In {\it Recent methods in nonlinear analysis and applications
 ({N}aples, 1980)}, pages 231--300. Liguori, Naples,1981.

\bibitem{Stuart1989}
C. A. Stuart, Bifurcation from the essential spectrum for some non-compact nonlinearities, {\it Math. Appl. Sci.}, {\bf11}(1989), no. 1, 525-542.
		
\bibitem{A20}
X. Yang, X. Tang, B. Chen, Multiple radial and nonradial normalized solutions for a quasilinear Schr\"odinger equation, {\it J. Math. Anal. Appl.}, {\bf501} (2021), 125122.		
		
\bibitem{A21}
H. Ye, Y. Yu, The existence of normalized solutions for $ L^{2} $-critical quasilinear Schr\"odinger equations, {\it J. Math. Anal. Appl.}, {\bf497}(2021), no.1, 124839.
		
\bibitem{A22}
L. Zhang, Y. Li, Z.-Q. Wang, Multiple normalized solutions for a quasilinear Schr\"odinger equation via dual approach, {\it  Topol. Methods Nonlinear Anal.}, {\bf61}(2023), 465-489.
\end{thebibliography}
\end{document}